\documentclass[12pt,oneside]{article}
\usepackage{amsmath,amssymb,amsfonts,amsthm}
\usepackage[utf8]{inputenc}
\usepackage{xcolor}
\newcommand{\T}{{\cal T}}

\newcommand{\set}[1]{\left\{#1\right\}}

\newcommand{\To}{\longrightarrow}

\newcommand {\cp}{\mathfrak{X}(\pi^{-1}(T M))}

\newtheorem{thm}{Theorem}[section]

\newtheorem{lem}[thm]{Lemma}
\newtheorem{prop}[thm]{Proposition}
\newtheorem{defn}[thm]{Definition}

\newtheorem{rem}[thm]{Remark}

\numberwithin{equation}{section}

\begin{document}
\title{{{\bf Golden Finsler Geometry: Local Properties and Global Deformations}}}
\author{\bf{Ebtsam H. Taha$^1$\footnote{Corresponding author}, Bankteshwar Tiwari$^2$ and A. Soleiman$^3$}}
\date{}
\maketitle                     
\vspace{-1.0cm}
\begin{center}
{ $^{1}$ Department of Mathematics, Faculty of Science, Cairo University, Giza,  Egypt
\vspace{0.2cm}
\\ $^{2}$ Centre for Interdisciplinary Mathematical Sciences, Institute of Science, Banaras Hindu University, Varanasi-221005, India}
\vspace{0.2cm}
\\$^{3}$ Department of Mathematics, Faculty of Science, Jouf University,  Skaka,  Kingdom of Saudi Arabia 
\end{center}
\vspace{-0.5cm}

\begin{center}
E-mails:$^{1}$ ebtsam@sci.cu.edu.eg, ebtsam.h.taha@hotmail.com \\
{\hspace{1.8cm}}
$^{2}$ btiwari@bhu.ac.in, banktesht@gmail.com\\
{\hspace{1.8cm}}$^{3}$ asoliman@ju.edu.sa,  amrsoleiman@yahoo.com\\
\end{center}

\vspace{0.7cm} \maketitle
\smallskip

 \noindent{\bf Abstract.} We introduce the concept of a golden Finsler structure on a finite-dimensional smooth manifold $M$ and investigate it from both local (coordinate-based) and global (coordinate-free) perspectives. Locally, we explicitly compute the fundamental metric tensor, establish the positive definiteness condition, and derive the geodesic spray coefficients. Furthermore, we investigate the projective flatness of the golden $(\alpha, \beta)$-metric and prove the non-existence of almost rational golden $(\alpha, \beta)$-metrics. Globally, we define the golden Finsler change $\widetilde{F}$ of a base Finsler metric $F$ and examine its geometric properties utilizing a special concurrent $\pi$-vector field. We explicitly determine how fundamental non-linear structures, including the Barthel and Berwald connections, transform under this change. Finally, we prove that $\widetilde{F}$ and $F$ cannot be projectively related.
 
\bigskip




\section{Introduction and Motivation}
\par The golden ratio has been recognized since ancient times and is observed in the pentagonal symmetry present in different flowers. This unique proportion also appears in musical compositions, sound frequencies, and human body dimensions. It has played a significant role throughout history in architecture and the visual arts. Notably, the golden ratio and the golden rectangle, defined by the sides of this ratio, can be found in the harmonious proportions of temples, churches, statues, paintings, and photographs. These pervasive appearances set the context for its profound mathematical and geometric value.

\bigskip

Mathematically, the golden ratio partitions a line segment into a major and a minor subsegment so that the ratio of the whole segment to the major subsegment is equal to the ratio of the major subsegment to the minor subsegment. Both ratios equal the number $\varphi$, which is the positive real root of the equation $x^2-x-1=0$ (thus, $\varphi=\frac{\sqrt{5}+1}{2}$). 
Moreover, the golden ratio is evident throughout scientific literature, where it plays a significant role across different fields, including atomic physics, Newtonian and relativistic mechanics, and the study of time dilation and Lorentz contraction in special relativity. It also features in the topology of four-manifolds, conformal field theory, probability theory, and Cantorian spacetime (see \cite{gold ratio3} for more details). This breadth of application underlines the ratio's deep mathematical importance.

\bigskip

The concept of a golden structure on a class of Riemannian manifolds was studied in \cite{gold ratio1, gold ratio2, gold ratio3}. More precisely, a $(1,1)$-tensor field $P$ on an $n$-dimensional Riemannian manifold $(M, g)$ is called a golden structure if it satisfies the equation $P^{2}=P+I$ (which mathematically mirrors the polynomial satisfied by the golden ratio $\varphi$), where $I$ stands for the identity tensor field. A Riemannian manifold endowed with such a structure is called a golden Riemannian manifold. 

\bigskip
\par
A Finsler structure on a smooth manifold extends the Riemannian metric in that its restriction to each tangent space is a norm rather than an inner product. For a Riemannian manifold $(M, \alpha)$, $\alpha=\sqrt{\alpha_{ij}(x)y^iy^j}$ and a 1-form $\beta=b_i(x)y^i$ on $M$, the $(\alpha, \beta)$-metrics $F:= \alpha \, \phi(s) := \alpha \,\phi \left(\frac{\beta}{\alpha}\right)$ provide an essential class of Finsler spaces (see, for example, \cite{Erasmo,  Projectively Flat Special, shenGCY, Soleiman-Taha_Mat, Tiwari-Kumar-Tayebi, [7]}).

\bigskip

In this paper, we first investigate the local coordinate-based properties of what we define as the \textit{golden $(\alpha , \beta )$-metric}, given by $F= \alpha \, \phi (s) := \alpha \, (s^2-s-1)$. We explicitly determine the coefficients of the Finsler metric tensor and its inverse, establish the positive definiteness condition, and derive the geodesic spray coefficients. Consequently, we investigate the projective flatness of this metric and prove the non-existence of almost rational golden $(\alpha, \beta)$-metrics. 

To extend our investigation globally, we utilize the pullback formalism of global Finsler geometry \cite{amr3, r86, [13]} to conduct a coordinate-free study. In the standard literature of $(\alpha, \beta)$-metrics, $\beta$ typically denotes the 1-form on $M$; however, because $\beta$ is already reserved for a different foundational concept in the pullback formalism, we adopt $\mathfrak{B}$ as our globally defined 1-form. 

By a \textit{golden Finsler structure} on a base Finsler manifold $(M,F)$, we mean the global transformation of the metric (not necessarily Riemannian) given by $\widetilde{F}=F\,\phi(s)$, with $s:=\frac {\mathfrak{B}} {F}$ and $\phi(s):=s^2-s-1$. The function $\widetilde{F}$ yields a conic-Finsler metric under certain conditions. Naturally, if the base Finsler metric $F$ is purely Riemannian (i.e., $F=\alpha$), this global geometric deformation identically recovers the local golden $(\alpha, \beta)$-metric. We explicitly derive the geometric objects associated with $\widetilde{F}$, including the fundamental metric and Cartan tensors.

\medskip

To make the global geometric derivations tractable and fully compute the geodesic spray and Berwald curvature of $\widetilde{F}$, we restrict our global study to a Finsler manifold $(M,F)$ that admits a concurrent $\pi$-vector field $\overline{A}$. We compute the corresponding $\pi$-form $\mathbf{B}:=i_{\overline{A}}\, g$, where $g$ is the metric tensor of $F$. Furthermore, the exact transformation of the curvature tensor field of the Barthel connection under this golden change is established. Consequently, we showed that the Finsler metrics $\widetilde{F}$ and $F$ cannot be projectively related. 

\medskip
This paper consists of four sections following the introduction. In \S 2, we recall the fundamental concepts of local Finsler geometry and the $(\alpha, \beta)$-metrics. Then,  we define the golden $(\alpha, \beta)$-metric and  derive its fundamental metric tensor, and geodesic spray coefficients. This enables us to establish a conditions for the projective flatness, and prove the non-existence of almost rational golden $(\alpha, \beta)$-metrics. In \S 3, we move to a global (coordinate-free) framework to introduce the golden Finsler change $\widetilde{F}$ of a Finsler metric $F$. We derive the coordinate-free expressions for different associated geometric objects such as the deformed fundamental metric and Cartan tensors, and establish the exact algebraic conditions required for non-degeneracy. Also, we compute the canonical geodesic spray of $\widetilde{F}$, determine the Barthel and Berwald connections, and conclusively prove that the base Finsler metric $F$ and its golden counterpart $\widetilde{F}$ are never projectively equivalent. The $\pi$-vector field $\overline{A}$ remains concurrent with respect to $\widetilde{F}$ exactly when the differential condition in equation \eqref{preserved concurrent} holds. Finally, we end the paper with a conclusion.
\section{Golden $(\alpha , \beta)$-metric: Local study}

In this section, we first explicitly compute the fundamental metric tensor and spray coefficients of the golden $(\alpha , \beta)$-metric. Next, we study the condition for its projective flatness. We end this section by examining the existence of an almost rational golden $(\alpha , \beta)$-metric on a smooth manifold.\\
\par Matsumoto introduced $(\alpha, \beta)$-metrics in \cite{MM} as follows: a Finsler metric $F=\alpha\, \phi \left(\frac{\beta}{\alpha}\right)=\alpha\, \phi(s)$  is called an $(\alpha, \beta)$-metric if $\phi$ is a smooth positive function defined on the open interval $(-a_{o}, a_{o})$ such that 
\begin{equation}\label{postive def. cond.}
\phi(s)- \phi(s)\, \phi'(s) +(b^2 -s^2 ) \,\phi''(s) > 0,\,\,\,\,\, |s| \leq \, b:=||\beta||_{\alpha} < a_{o}. 
\end{equation}
\begin{defn}
The golden $(\alpha , \beta)$-metric is an $(\alpha , \beta)$-metric $F = \alpha \phi (s)$ with $$\phi(s):=s^2-s-1, \quad s=\frac{\beta}{\alpha}.$$ 
\end{defn}
\begin{lem}
A golden $(\alpha, \beta)$-metric $F=\alpha (s^2 - s - 1)$ with $b > 1/\sqrt{2}$ is  a Finsler metric on the open subset  of the slit tangent bundle $\T M$ which is the domain of $s$:
$$\mathcal{D} = \left(-k, \frac{1-\sqrt{5}}{2}\right) \cup \left(\frac{1+\sqrt{5}}{2}, k\right),\quad k = \sqrt{\frac{2b^2-1}{3}}.$$
\end{lem}
\begin{proof}
In view of \eqref{postive def. cond.},  an $(\alpha, \beta)$-metric is a Finsler metric if $\phi(s) > 0$, and the strong convexity condition $\phi(s) - s\phi'(s) + (b^2 - s^2)\phi''(s) > 0$. 

For the golden $(\alpha, \beta)$-metric,  $\phi(s) = s^2 - s - 1 > 0$ holds when $s < \frac{1-\sqrt{5}}{2}$ or $s > \frac{1+\sqrt{5}}{2}$. The strong convexity simplifies to $2b^2 - 3s^2 - 1 > 0$, or equivalently,  $s^2 < \frac{2b^2-1}{3}$. By defining $k = \sqrt{\frac{2b^2-1}{3}}$, this condition is satisfied whenever $s \in (-k, k)$.  For $k \in \mathbb{R}$,   we require $2b^2 - 1 > 0$, which is satisfied by forcing $b > 1/\sqrt{2}$. Therefore,  $$s \in \mathcal{D}=\left(-k, \frac{1-\sqrt{5}}{2}\right) \cup \left(\frac{1+\sqrt{5}}{2}, k\right).$$
Consequently, on $\mathcal{D}$ the golden $(\alpha, \beta)$-metric fulfils all requirements for being a  Finsler metric. Thus, while the golden $(\alpha, \beta)$-metric is not regular globally for all $|s| < b$, it is locally regular within this precisely defined domain $\mathcal{D}$.
\end{proof}
\subsection{Fundamental metric tensor and  spray coefficients} 
 It is known that the fundamental Finsler metric tensor of an $(\alpha , \beta)$-metric is given by (see, e.g. \cite{[7]})
   \begin{equation} \label{gen metric}
   g_{ij}=\rho_{0}\, \alpha_{ij} + \rho_{1}\, b_{i}\, b_{j}+ \frac{\rho_{2}}{\alpha} \, \left(b_{i} y_{j} + b_{j} y_{i} -  \frac{s}{\alpha} y_{i} y_{j}\right),
\end{equation}   
 where $y_{j}:= \alpha_{ij}\, y^{i} ,$
    $$\rho_{0}:=  \left(\phi^2 -s \phi\, \phi' \right),\,\, \,\rho_{1}:= \left( \phi \, \phi'' + (\phi')^{2}\right), \,\,\, \rho_{2}:= \left(\phi\, \phi' -s\left[\phi\,\phi'' + (\phi')^{2} \right] \right).$$
The determinant of the fundamental Finsler metric tensor of an $(\alpha, \beta)$-metric is \begin{equation}\label{determinant}
\det (g_{ij})= \phi ^{n+1} \left(\phi -s  \phi' \right)^{n-2}  \left(\phi -s  \phi' + (b^2 -s^2 )\phi'' \right)\det (\alpha_{ij}).
\end{equation}
In addition, the inverse of the fundamental Finsler metric is given by
\begin{equation}\label{inverse}
g^{ij}=\rho_0 ^{-1}\, \left\{\alpha^{ij} + \eta\, b^{i}\, b^{j}+ \eta_{0}\,{\alpha}^{-1} \, \left(b^{i} y^{j} + b^{j} y^{i} \right) + \frac{\eta _{1}}{\alpha ^2} y^{i} y^{j} \right\},
\end{equation}
where 
\begin{eqnarray*}
\eta &=&  \frac{-2\phi ''}{2 \left(\left(b^2-s^2\right) \phi ''-s \phi '+\phi \right)},\\
\eta_{0} &=& \frac{- \phi  \phi '+s \left(\phi  \phi ''+\phi '^2\right)}{ \phi \left(\left(b^2-s^2\right) \phi ''-s \phi '+\phi \right)},\\
\eta_{1} &=& \frac{\left(\left(b^2-s^2\right) \phi '+s \phi \right) \left(\left(\phi -s \phi '\right) \phi '-s \phi \, \phi ''\right)}{ \phi^{2}  \left(\left(b^2-s^2\right) \phi ''-s \phi '+\phi \right)},
\end{eqnarray*}
and $\alpha^{ij}$ are the component of inverse matrix of $\alpha_{ij}$.
\begin{prop} The fundamental Finsler metric tensor, inverse of the fundamental Finsler metric tensor and the determinant of the fundamental Finsler metric tensor for the golden-$(\alpha, \beta)$ metric are given, respectively, by:
 \begin{eqnarray} \label{golden alpha,beta metric}
   g_{ij}&=&(1+s+s^3-s^4 )\, \alpha_{ij} + (6 s^2 -6s -1 )\, b_{i}\, b_{j} \\ \nonumber &&
   + \frac{(1+ 3 s^2 - 4 s^3 )}{\alpha} \, \left(b_{i} y_{j} + b_{j} y_{i} -  \frac{s}{\alpha} y_{i} y_{j}\right),
\end{eqnarray} 
\begin{eqnarray}
g^{ij} =\frac{1}{1+s+s^3-s^4}\left\{\alpha^{ij} + \eta\, b^{i}\, b^{j}+ \eta_{0}\,{\alpha}^{-1} \, \left(b^{i} y^{j} + b^{j} y^{i} \right) + \frac{\eta _{1}}{\alpha ^2} y^{i} y^{j} \right\}
\end{eqnarray}
\begin{equation}\label{determinant of golden}
\det (g_{ij})= (-1)^{n-2} \,(s^2-s-1)^{n+1} \left(s^2+1 \right)^{n-2}  \left(2b^2 -3s^2 -1 \right)\det (\alpha_{ij}).
\end{equation}
\end{prop}
\begin{proof}
    By direct calculations, $\phi'(s):=2s-1,\,\, \phi''(s):=2$, which implies
$$\rho_{0}= 1+s+s^3-s^4 ,\,\quad \rho_{1}=6 s^2 -6s -1 ,\,\quad \rho_{2}= 1+ 3 s^2 - 4 s^3.$$ Therefore, we obtain \eqref{golden alpha,beta metric}. Consequently,  the determinant of the  fundamental Finsler metric tensor of the golden-$(\alpha, \beta)$ metric is given by \eqref{determinant of golden}. 
Since
\begin{eqnarray*}
\eta (s) &=&  \frac{-2}{2b^2 -3s^2 -1},\\
\eta_{0} (s) &=&  -\frac{(s-1) \left(4 s^2+s+1\right)}{ ((s-1) s-1) \left(-2 b^2+3 s^2+1\right)},\\
\eta_{1}(s)&=& -\frac{(s-1) \left(4 s^2+s+1\right) \left(b^2 (1-2 s)+s^3+s\right)}{((s-1) s-1)^2 \left(-2 b^2+3 s^2+1\right)}.
\end{eqnarray*}
Now, plugging the expressions of $\eta,\,\eta_{0},\,\eta_{1} $ into \eqref{inverse}, we obtain the following
\begin{eqnarray*}
g^{ij}&=&\rho_0 ^{-1}\, \left\{\alpha^{ij} - \frac{2}{2b^2 -3s^2 -1}\, b^{i}\, b^{j}-\frac{(s-1) \left(4 s^2+s+1\right)}{\alpha ((s-1) s-1) \left(-2 b^2+3 s^2+1\right)} \, \left(b^{i} y^{j} + b^{j} y^{i} \right)\right.\\
&& \qquad \quad \left.-\frac{1}{\alpha ^2} \frac{(s-1) \left(4 s^2+s+1\right) \left(b^2 (1-2 s)+s^3+s\right)}{((s-1) s-1)^2 \left(-2 b^2+3 s^2+1\right)}y^{i} y^{j} \right\}.
\end{eqnarray*}
Let $z(s):= 2b^2 -3s^2 -1$. Thus, we get
\begin{eqnarray*}
g^{ij}&=&\rho_0 ^{-1}\, \left\{\alpha^{ij} - \frac{2}{z(s)}\, b^{i}\, b^{j}-\frac{(s-1) \left(4 s^2+s+1\right)}{\alpha ((s-1) s-1) z(s)} \, \left(b^{i} y^{j} + b^{j} y^{i} \right)\right.\\
&& \qquad \quad \left.+\frac{1}{\alpha ^2} \frac{(s-1) \left(4 s^2+s+1\right) \left(b^2 (1-2 s)+s^3+s\right)}{((s-1) s-1)^2 z(s)}y^{i} y^{j} \right\}.
\end{eqnarray*}
\end{proof}
\begin{rem}
As  $s,\,\phi (s) \neq 0$ and $(s^2+1) \neq 0$, then
the golden-$(\alpha, \beta)$ metric is non-degenerate if and only if $$z(s):=(-1+2b^2 -3s^2) \neq 0.$$
In addition, the positive definite condition for the golden-$(\alpha, \beta)$ metric is 
$$ z(s)> 0.$$
For $\alpha \neq 0$,  it  is equivalent to, $$ \alpha ^2 (-1+2b^2) - 3 \beta ^2 >0. $$
\end{rem}
\begin{prop} The spray coefficients $G^i$ of the golden $(\alpha , \beta)$-metric are given by (\ref{spray1}), where the scalars $Q(s),\Theta(s)$, and $\Psi(s)$ are given by (\ref{spray2}), (\ref{spray3}) and (\ref{spray4}) respectively. 
\end{prop}
\begin{proof} It is known that the spray coefficients of an $(\alpha , \beta)$-metric are given by (see, e.g.,  \cite{shenGCY})
\begin{equation}{\label{spray1}}
G^{i} = G^{i}_{\alpha} + \alpha \, Q \, s^{i}_{0} + (-2 \alpha \, Q \, s_{0} +r_{00})( \Psi \, b^{i} + \Theta \, \alpha ^{-1} \, y^{i}),
\end{equation}
where 
\begin{eqnarray*}
Q(s)&:= & \frac{\phi '(s)}{\phi (s)-s \phi '(s)}\\
 \Theta (s) &:= & \frac{\phi (s) \phi '(s)-s \left(\phi (s) \phi ''(s)+\phi '(s)^2\right)}{2 \phi (s) \left(\left(b^2-s^2\right) \phi ''(s)-s \phi '(s)+\phi (s)\right)}\\
 \Psi (s) &:= & \frac{\phi ''(s)}{2 \left(\left(b^2-s^2\right) \phi ''(s)-s \phi '(s)+\phi (s)\right)}.
\end{eqnarray*}
For a golden $(\alpha, \beta)$-metric, we have
\begin{eqnarray} \label{spray2}
Q(s)&= & \frac{1-2 s}{s^2+1}\\ \label{spray3}
 \Theta (s) &= &- \frac{(s-1) \left(4 s^2+s+1\right)}{2 ((s-1) s-1) z(s)}\\ \label{spray4}
 \Psi (s) &= & \frac{1}{z(s)}
\end{eqnarray}
It should be noted that $\eta (s) = -2 \Psi (s) ,\quad
\eta_{0} (s) = -2 \Theta (s).$ \end{proof}
\subsection{Projective flatness of the Golden $(\alpha, \beta)$-metric}
\begin{defn}
A Finsler manifold is said to be projectively flat if there exists a local coordinate system in which its geodesics are line segments. 
\end{defn}
It is an interesting problem in Finsler geometry to study and characterize projectively 
flat Finsler metrics. The following Lemma is well known:
\begin{lem}\emph{\cite{shenGCY}}
    An $(\alpha , \beta)$-metric $F=\alpha\, \phi(s)$ is projectively flat in an open subset $U$ of $\mathbb R^n$ if and only if
    \begin{equation}\label{lmm1}
        (\alpha_{mi}\alpha^2-y_my_i)G^m_{\alpha}+\alpha^3Qs_{i0}+\Psi\alpha(-2\alpha Qs_0+r_{00})(b_i\alpha-sy_i)=0.
    \end{equation}
    Here $r_{ij}=\frac{1}{2}(b_{i|j}+b_{j|i})$, $s_{ij}=\frac{1}{2}(b_{i|j}-b_{j|i})$, $s_j=s_{ij}b^i$ and the suffix $`0'$ indicates contraction with $y^i$, for instance, $s_0=s_iy^i$.
\end{lem}

In view of (\ref{spray2}) and (\ref{spray4}), for the Golden $(\alpha, \beta)$-metric, equation(\ref{lmm1}) can be written as
\begin{equation}\label{lmm2}
        (\alpha_{mi}\alpha^2-y_my_i)G^m_{\alpha}+\alpha^3\frac{1-2 s}{s^2+1}s_{i0}-\frac{\alpha}{z(s)}(2\alpha s_0\frac{1-2 s}{s^2+1}-r_{00})(b_i\alpha-sy_i)=0.
    \end{equation}
By the above identity, we prove the following:
\begin{thm}
    The golden $(\alpha, \beta)$-metric is projectively flat if and only if
    \begin{enumerate}
        \item[(i)] $b_{i|j}=\tau((2 b^2-1)\alpha_{ij}-3b_ib_j)$
        \item[(ii)] $G^i_{\alpha}=\theta y^i-\tau \alpha^2b^i$, where $\tau=\tau(x)$ and $\theta=\theta_i(x)y^i$.
    \end{enumerate}
     In this case
    \begin{equation}\label{G} G^i=(\theta + \tau \chi \alpha)y^i,\end{equation} 
    where \begin{equation}\label{chi}
    \chi:=\frac{(1-2s)(1+s^2)}{2(s^2-s-1)}-s .\end{equation}
\end{thm}
\begin{proof}
First, we rewrite (\ref{lmm2}) as a polynomial in $y$ and $\alpha$. This gives
\begin{equation}\label{proj2}\begin{split}
(\alpha_{mi}\alpha^2-y_my_i)G^m_{\alpha}(\alpha^2+\beta^2)[(2 b^2-1)\alpha^2-3\beta^2]+ 
\alpha^4(\alpha-2 \beta)[(2 b^2-1)\alpha^2-3\beta^2]s_{i0}\\
+\alpha^2[-2\alpha^2 s_0(\alpha-2 \beta)+r_{00}(\alpha^2+\beta^2)](b_i\alpha^2-\beta y_i)=0.\end{split}\end{equation}
The coefficients of the odd powers of $\alpha$ must be zero, as they contain irrational powers of $y$. We obtain
\begin{equation} \label{proj1}
[(2 b^2-1)\alpha^2-3\beta^2]s_{i0}=2s_0(b_i\alpha^2-\beta y_i)
\end{equation}
Contracting by $b^i$, the above equation reduces to
$$s_0(\alpha^2+\beta^2)=0,$$
which gives $s_0=0$ (since, $1+ s^2 \neq 0$). Then, as $(2 b^2-1)\alpha^2-3\beta^2 \neq 0$, it follows from (\ref{proj1}) that
\begin{equation}\label{proj03}
    s_{i0}=0
\end{equation}
Thus, $\beta$ is closed. Now equation (\ref{proj2}) is reduced to
\begin{eqnarray*}
(\alpha_{mi}\alpha^2-y_my_i)G^m_{\alpha}[(2 b^2-1)\alpha^2-3\beta^2]+\alpha^2r_{00}(b_i\alpha^2-\beta y_i)=0.\end{eqnarray*}
Contracting the above equation by $b^i$, we have
$$G^m_{\alpha}(b_m\alpha^2-\beta y_m)[(2 b^2-1)\alpha^2-3\beta^2]+\alpha^2r_{00}(b^2\alpha^2-\beta^2)=0.$$
Since $(2 b^2-1)\alpha^2-3\beta^2$ is neither divisible by $\alpha^2$ nor by $(b^2\alpha^2-\beta^2)$, therefore \\$G^m_{\alpha}(b_m\alpha^2-\beta y_m)$ is divisible by 
$\alpha^2(b^2\alpha^2-\beta^2)$. Therefore, there is a smooth function $\tau $ in $M$ such that
\begin{equation}\label{proj3}
    r_{00}=\tau((2 b^2-1)\alpha^2-3\beta^2).
\end{equation}
Since $s_{i0}=s_{ij}y^j=\frac{1}{2}(b_{i|j}-b_{j|i})y^j$, 
(\ref{proj03}) can be written as $b_{i|0}=b_{0|i}$.
Furthermore, utilizing $r_{00}=r_{ij}y^iy^j=\frac{1}{2}(b_{i|j}+b_{j|i})y^iy^j$ and $b_{i|0}=b_{0|i}$, equation (\ref{proj3}) can be expressed as
$$b_{i|j}=\tau((2 b^2-1)\alpha_{ij}-3b_ib_j).$$
\noindent By (\ref{proj2}) and (\ref{proj3}), the formula (\ref{spray1}) for $G^i$ can be simplified to 
\begin{equation}\label{proj4}
    G^i=G^i_{\alpha}+\tau \chi \alpha y^i+\tau \alpha^2 b^i,
\end{equation}
where $\chi $ is given in (\ref{chi}). We know that $F$ is projectively flat if and only if there exists a positively homogeneous function $P$ in $y$ such that $G^i (x,y) = P (x,y) \,y^i$.
By (\ref{proj4}), this is equivalent to 
$$G^i_{\alpha}=\theta y^i-\tau \alpha^2b^i,$$
where $\theta$  is a $1$-form on $M$ given by $\theta = \theta_{i}(x) y^i$. In this case, the geodesic spray coefficients of the golden $(\alpha , \beta)$-metric $G^i$ are given by (\ref{G}).
\end{proof}

\subsection{Almost rational golden $(\alpha, \beta)$-metric }
Motivated by the comments of D. Bao \cite[pages: 19,20]{Bao}, recently, Taha and Tiwari \cite{TT} defined the almost rational Finsler metrics as follows:
\begin{defn}
A Finsler metric $F$ on a manifold $M$ is said to be an almost rational Finsler metric if the components $g_{ij}(x,y)$ of its Finsler metric tensor are  expressed in the form
\begin{equation} \label{gdef}
g_{ij}(x,y)= \eta(x,y)\, a_{ij}(x,y),
\end{equation}
where \begin{itemize}
\item[(a)] $\eta: TM \longrightarrow (0,\infty )$ is a smooth function,
\item[(b)] the matrix $( a_{ij}(x,y))_{1 \leq i,j \leq n}$ is  symmetric positive definite,
\item[(c)]  for each $i,j,\, a_{ij}(x,y)$ is a  rational function in the fiber coordinate $y$.
\end{itemize}
The pair $(M,F)$ is said to be  an almost rational-Finsler manifold. If, in addition,  $\eta$ is a rational function in $y$, we call the pair $(M,F)$ a rational Finsler manifold.
\end{defn}

It turns out that there are several well-known Finsler metrics, for instance, the Kropina metric, the generalized Kropina metric, which are almost rational Finsler metrics, while the Randers metric is not an almost rational Finsler metric. In the following proposition, we prove that this is also the case with the golden $(\alpha, \beta)$-metric. More precisely, we have

\begin{thm} There exists no almost rational golden $(\alpha, \beta)$-metric.
\end{thm}
\begin{proof} 
The golden $(\alpha , \beta)$-metric is given by  $F = \alpha \phi (s)$ with
 $\phi(s):=s^2-s-1$. In view of equation (\ref{golden alpha,beta metric}), its metric tensor is given by

\begin{eqnarray} \label{golden alpha,beta metric again}
   g_{ij}&=& \frac{\alpha ^{4} + \beta \, \alpha ^{3} + \beta ^{3} \, \alpha - \beta ^{4}}{\alpha ^{4}}\, \alpha_{ij} + \frac{6\beta ^2 - 6 \alpha \beta - \alpha ^{2}}{\alpha ^{2}}\, b_{i}\, b_{j} \\ \nonumber &&
   + \frac{(\alpha ^3 + 3 \alpha \, \beta ^2 - 4 \beta ^3 )}{\alpha ^{4}} \, \left(b_{i} y_{j} + b_{j} y_{i} -   \frac{\beta}{\alpha ^2 }y_{i} y_{j}\right),
\end{eqnarray} 
This can be rewritten as 
\begin{eqnarray} \label{golden alpha,beta metric 2}
   g_{ij}&=& A_{ij} + \alpha B_{ij} \end{eqnarray} 
where $A_{ij}$ and $B_{ij}$ are rational functions of $y$ given by
\begin{eqnarray} \label{golden alpha,beta metric 3}A_{ij}=\frac{\alpha ^{4} - \beta ^{4}}{\alpha ^{4}}\, \alpha_{ij} + \frac{6\beta ^2  - \alpha ^{2}}{\alpha ^{2}}\, b_{i}\, b_{j} 
   - \frac{4 \beta ^3 }{\alpha ^{4}} \, \left(b_{i} y_{j} + b_{j} y_{i} -   \frac{\beta}{\alpha ^2 }y_{i} y_{j}\right),\end{eqnarray}
and 
\begin{eqnarray} \label{golden alpha,beta metric 4}B_{ij}= \frac{ \beta \, \alpha ^{2} + \beta ^{3} }{\alpha ^{4}}\, \alpha_{ij} - \frac{6  \beta }{\alpha ^{2}}\, b_{i}\, b_{j} 
   + \frac{(\alpha ^2 + 3  \beta ^2  )}{\alpha ^{4}} \, \left(b_{i} y_{j} + b_{j} y_{i} -   \frac{\beta}{\alpha ^2 }y_{i} y_{j}\right).\end{eqnarray}
In view of equation (\ref{golden alpha,beta metric 2}), the golden Finsler metric is almost rational if either $A_{ij} $ or $B_{ij} $ vanishes identically.

If possible, let $A_{ij} =0$. Contracting equation (\ref{golden alpha,beta metric 3}) with $y^iy^j$, and taking into account $\alpha ^2 = \alpha_{ij} y^i y^j ,\, \beta = b_i y^i$, we have
\begin{eqnarray*}
y^iy^jA_{ij}&=&\frac{\alpha ^{4} - \beta ^{4}}{\alpha ^{4}}\, \alpha_{ij}y^iy^j + \frac{6\beta ^2  - \alpha ^{2}}{\alpha ^{2}}\, b_{i}\, b_{j} y^iy^j \\
&& - \frac{4 \beta ^3 }{\alpha ^{4}} \, \left(b_{i} y_{j} + b_{j} y_{i} -   \frac{\beta}{\alpha ^2 }y_{i} y_{j}\right)y^iy^j \\ &=&
 \frac{\alpha ^{4} - \beta ^{4}}{\alpha ^{4}}\, \alpha^2 + \frac{6\beta ^2  - \alpha ^{2}}{\alpha ^{2}}\, \beta^2
   - \frac{4 \beta ^3 }{\alpha ^{4}} \, \left(2\beta \alpha^2 -   \frac{\beta}{\alpha ^2 }\alpha^4\right)  \\ &=&  
\frac{\alpha ^{4} - \beta ^{4}}{\alpha ^{2}} + \frac{6\beta ^4  - \alpha ^{2}\beta^2}{\alpha ^{2}} 
   - \frac{4 \beta ^4 }{\alpha ^{2}}     \\    &=&  
 \frac{\alpha ^{4} - \beta ^{4} + 6\beta ^4  - \alpha ^{2} \beta^2 - 4\beta^4 }{\alpha ^{2}} \\ &=& \frac{\alpha ^{4}  - \alpha ^{2} \beta^2 + \beta^4 }{\alpha ^{2}}= 0 .      
\end{eqnarray*}
That is, $\alpha^4 - \alpha^2\beta^2 + \beta^4 = 0$. Treating this as a quadratic equation in terms of $\alpha^2$, the discriminant is $\Delta = (-\beta^2)^2 - 4(1)(\beta^4) = -3\beta^4 < 0$. This implies there are no real roots for $\alpha^2$, which is a contradiction since the metric scalar $\alpha$ must be real and positive.

Now, if $B_{ij} =0$, contracting equation (\ref{golden alpha,beta metric 4}) with $y^iy^j$ yields $\frac{2\beta(\alpha^2 - \beta^2)}{\alpha^2}=0$. Assuming $\beta \neq 0$ (otherwise $F$ reduces to a purely Riemannian metric), this gives $\alpha^2 = \beta^2$. This implies $\alpha_{ij} y^i y^j = b_i b_j y^i y^j$, meaning the Riemannian metric tensor $\alpha_{ij} = b_i b_j$. This is a contradiction, as the outer product $b_i b_j$ has a rank of 1, whereas a positive-definite Riemannian metric must have full rank $n \geq 2$.

Hence, there is no golden $(\alpha, \beta)$-metric whose Finsler metric tensor can be written in the form $\eqref{gdef}$.
\end{proof}

\section{Golden Finsler change: Global study}
\subsection{Preliminaries}
Let $M$ be an $n$-dimensional smooth manifold and $(TM, \pi, M)$ be the tangent bundle with the differential $d\pi: TTM \longrightarrow TM$ and slit tangent bundle $\T M$. The vertical bundle $V(TM)$ of $TM$ is defined by $\ker(d\pi)$, and the pullback bundle of the tangent bundle is denoted by $\pi^{-1}(TM)$. We use the notation $C^\infty(TM)$ for the algebra of $C^\infty$ functions on $TM$ and $\cp$ the $C^\infty(TM)$-module of differentiable sections of the pullback bundle $\pi^{-1}(TM)$. We call the elements of $\cp$ $\pi$-vector fields and will be denoted by barred letters $\overline{X}$. 


We have the short exact sequence of vector bundles \cite{r21, Soleiman-Taha_Mat}

$$
0 \longrightarrow \pi^{-1}(T M) \stackrel{\gamma}\longrightarrow T{TM} \stackrel{\rho}\longrightarrow \pi^{-1}(T M) \longrightarrow 0,
$$
where $\gamma$ is the natural injection and $\rho := (\pi, d\pi)$.

\par Let us recall some basics and properties of the Klein-Grifone approach to Finsler geometry. We refer to \cite{r21, r22, r27} for more details. The tangent structure (or the vertical endomorphism) $J$ of $TM$ is defined by $J = \gamma \circ \rho$. The Liouville vector field (or the fundamental $\pi$-vector field) $\mathcal{C}$ is defined by $\mathcal{C} := \gamma \overline{\eta}$, where $\overline{\eta}(u) = (u,u)$ for all $u \in \T M$.
\begin{defn}\cite{Finsler definitions, Lovas, r94a}
A Finsler metric (or, Finsler structure) $F$ on $M$ is a map $F: TM \To [0,\infty)$
 such that $F $  is smooth on  $\T M$,  continuous on $TM$,
     homogeneous of degree $1$ in the directional argument $y$ (i.e., $\mathcal{L}_{\mathcal{C}} F=F$, and the exterior $2$-form
    $\Omega:=\frac{1}{2} dd_{J}F^2$  has maximal rank.
    The Finsler metric tensor $g$ induced by $F$ on $\pi^{-1}(TM)$  is defined by
\begin{equation}\label{g}
g(\rho X,\rho Y):=\Omega(JX,Y), \ \forall  X, Y \in
    \mathfrak{X}(TM).
\end{equation}
We will use the notation $(M,F)$ for a Finsler manifold. Moreover, if $F$  satisfies the above conditions on a conic  subset $\mathfrak{U}$ of $\T M$ (that is,   if $p\in \mathfrak{U}$ and $\lambda>0$, then $\lambda p \in \mathfrak{U}$), then  $(M,F)$ is  called a conic Finsler manifold.
 \end{defn}
For a Finsler manifold $(M, F)$, the associated canonical spray is vector field $G$ on $TM$ which is $C^{1}$,   smooth on $\T M$ such that $JG = \mathcal{C}$, $[\mathcal{C}, G] = G$, and satisfies $i_{G}dd_J F^2 = -d F^2$. Consequently, the Barthel connection $\Gamma$ defined by $\Gamma = [J, G]$ has horizontal and vertical projectors which are given, respectively, by $h := \frac{1}{2}(Id + \Gamma)$ and $v := \frac{1}{2}(Id - \Gamma)$. The curvature of $\Gamma$ is defined by $\mathfrak{R} := -\frac{1}{2}[h, h]$. \\

For a regular linear connection $D$ on $\pi^{-1}(TM)$, the associated connection map $K$ is defined by \vspace{-0.1cm} $$K: TTM \longrightarrow \pi^{-1}(TM): X \longmapsto D_X \overline{\eta},$$ and the horizontal space $H_u(TM)$ at $u$ is $H_u(TM) := \{ X \in T_u(TM) : K(X) = 0 \}$. Consequently,
$$
T_u(TM) = V_u(TM) \oplus H_u(TM) \quad \forall \, u \in TM.
$$

Moreover, the vector bundle maps $\rho |_{H(TM)}$ and $K |_{V(TM)}$ are isomorphisms. The map $\beta := (\rho |_{H(TM)})^{-1}$ is called the horizontal map of $D$. Hence, $\rho \circ \beta = \beta \circ \rho = Id$. \\

Two well-known linear connections in the context of Finsler geometry are used in the sequel, namely, Cartan and Berwald connections on $\pi^{-1}(TM)$. The Cartan connection $\nabla$ is metric (i.e., $\nabla g = 0$, or equivalently, $\nabla_{\gamma \overline{X}}g = 0$, $\nabla_{\beta \overline{X}}g = 0$) with vanishing (h)h-torsion $Q = 0$, and its (h)hv-torsion $T$ satisfies $g(T(\overline{X}, \overline{Y}), \overline{Z}) = g(T(\overline{X}, \overline{Z}), \overline{Y})$. On the other hand, the metricity of Berwald connections is characterized by \cite{r92}
$$
(D^{\circ}_{\gamma \overline{X}}g)(\overline{Y}, \overline{Z}) = 2\mathbf{T}(\overline{X}, \overline{Y}, \overline{Z}), \quad (D^{\circ}_{\beta \overline{X}}g)(\overline{Y}, \overline{Z}) = -2\widehat{\mathbf{P}}(\overline{X}, \overline{Y}, \overline{Z}),
$$
where $\widehat{\mathbf{P}}(\overline{X}, \overline{Y}, \overline{Z}) := g(\widehat{P}(\overline{X}, \overline{Y}), \overline{Z})$ and $\widehat{P}$ is the (v)hv-torsion tensor of the Cartan connection. For more details about the pullback approach to global Finsler geometry, we refer, for example, to \cite{r93, amr3, Szilasi-book, r94, r96}.

\begin{lem}\cite{Soleiman-Taha_Mat}\label{B1} 
Given a Finsler metric $F$, one can obtain
\begin{description}
    \item[(a)] $D^{\circ}_{\gamma \overline{X}}F = dF(\gamma \overline{X}) = d_{J}F(\beta \overline{X}) = \ell(\overline{X}), \quad d_{J}F(\gamma \overline{X}) = 0, \quad \ell := F^{-1}i_{\overline{\eta}}g.$
    \item[(b)] $D^{\circ}_{\beta \overline{X}}F = d_{h}F(\beta \overline{X}) = dF(\beta \overline{X}) = 0.$
    \item[(c)] $\hbar(\overline{X}, \overline{Y}) = (D^{\circ}_{\gamma \overline{X}}\ell)(\overline{Y})F = (\nabla_{\gamma \overline{X}}\ell)(\overline{Y})F, \quad \Omega(\gamma \overline{X}, \beta \overline{Y}) = g(\overline{X}, \overline{Y}).$
\end{description}
\end{lem}
\subsection{Golden Finsler structure}
Now, an intrinsic study of the golden Finsler structure, as well as that of the golden $(\alpha, \beta)$-metric, is underway. 

\begin{defn}\cite{r94a}\label{ch5def.ind} Let $(M,F)$ be a Finsler manifold. A $\pi$-vector field $\overline{W} \in \cp$ (resp. a $\pi$-form $\Theta$) is said to be independent of the fiber coordinates $y$ if $D^{\circ}_{\gamma \overline{U}}\overline{W}=0$ (resp. $D^{\circ}_{\gamma \overline{U}}\, \Theta=0$) for all $\overline{U} \in \cp$.
\end{defn}

\begin{defn} For a Finsler manifold $(M,F)$, let $\mathbf{B}$ be a scalar $\pi$-form independent of the directional argument $y$, and let $\mathfrak{B}$ be a smooth function on $TM$ defined by ${\mathfrak{B}}(x,y) := \mathbf{B}(\overline{\eta})$. The Finsler structure $F$ is changed to $\widetilde{F}$ as follows:
\begin{equation}\label{change}
\widetilde{F}(x,y) = \frac{\mathfrak{B}^2(x,y)}{F(x,y)} - F(x,y) - \mathfrak{B}(x,y).
\end{equation}
If $\widetilde{F}$ defines a new Finsler structure on $M$, then $\widetilde{F}$ will be referred to as a golden Finsler metric.
\end{defn}

\begin{rem} 
It is clear that the golden Finsler metric $\widetilde{F}(x,y) = F\phi(s)$ with $s := \frac{\mathfrak{B}}{F}$ and $\phi(s) := s^2 - s - 1$, whose roots are $s_1 = \frac{1}{2} \left(1-\sqrt{5}\right)$ and $s_2 = \frac{1}{2} \left(1+\sqrt{5}\right)$.
Moreover, the golden Finsler structure $\widetilde{F}$ is positive (i.e., $\widetilde{F} > 0$) for 
$$ 
s < \frac{1}{2} \left(1-\sqrt{5}\right) \quad \text{or} \quad s > \frac{1}{2} \left(1+\sqrt{5}\right).
$$
\end{rem}

\begin{lem}\label{B} 
\begin{description}
    \item[(i)] Under the golden Finsler change $\widetilde{F}$ of a Finsler metric $F$, the vertical counterpart of the Berwald connection ${{D}}^{\circ}_{\gamma \overline{X}} \overline{Y}$ remains invariant, that is,
    ${{\widetilde{D}}}^{\circ}_{\gamma \overline{X}}\, \overline{Y} = {{D}}^{\circ}_{\gamma \overline{X}}\, \overline{Y}.$
    
    \item[(ii)] Let $(M,F)$ be a Finsler manifold equipped with a scalar $\pi$-form $\mathbf{B}$ which is independent of the directional argument $y$, and $\overline{A}$ its associated $\pi$-vector field given by $i_{\overline{A}}\,g := \mathbf{B}$. Then, the function $\mathfrak{B}(x,y)$ has the following properties:
    \begin{description}
        \item[(a)] $d_{J}\mathfrak{B}(\gamma \overline{X}) = 0, \quad {{D}}^{\circ}_{\gamma \overline{X}}\mathfrak{B} = d\mathfrak{B}(\gamma \overline{X}) = d_{J}\mathfrak{B}(\beta \overline{X}) = \mathbf{B}(\overline{X})$.
        
        \item[(b)] $d_{h}\mathfrak{B}(\beta \overline{X}) = {{D}}^{\circ}_{\beta \overline{X}}\mathfrak{B} = d\mathfrak{B}(\beta \overline{X}) = F\,\ell({{D}}^{\circ}_{\beta \overline{X}}\overline{A}), \quad d\mathfrak{B}(G) = F\,\ell({{D}}^{\circ}_{G}\overline{A})$.
        
        \item[(c)] ${{D}}^{\circ}_{\gamma \overline{X}}\overline{A} = -2{T}(\overline{X},\overline{A})$.
    \end{description}
\end{description}
\end{lem}

\begin{proof} 
The proof is similar to the corresponding results in \cite{r92, Soleiman-Taha_Mat}.
\end{proof}
\begin{prop}\label{hh1} 
For the golden Finsler metric {\em \eqref{change}}, we obtain the following:
\begin{description} 
    \item[(a)] its fundamental Finsler metric tensor $\widetilde{g}$ is given by
    \begin{align*}
    \widetilde{g}(\overline{X},\overline{Y}) &= C_1(F, \mathfrak{B})\,g(\overline{X},\overline{Y})
     -\frac{\mathfrak{B}}{F}C_2(F, \mathfrak{B})\,\ell(\overline{X})\,\ell(\overline{Y}) + C_3(F, \mathfrak{B})\,\mathbf{B}(\overline{X})\,\mathbf{B}(\overline{Y})\\
    &\quad +C_2(F, \mathfrak{B})\,\set{\mathbf{B}(\overline{X})\,\ell(\overline{Y})+\mathbf{B}(\overline{Y})\,\ell(\overline{X})},
    \end{align*}
    where 
    \begin{align*}
        C_1(F, \mathfrak{B}) &:= \frac{\left(\mathfrak{B}^2+F^2\right) \left(-\mathfrak{B}^2+\mathfrak{B} F+F^2\right)}{F^4}, \\
        C_2(F, \mathfrak{B}) &:= \frac{F^3-4 \mathfrak{B}^3+3 \mathfrak{B}^2 F}{F^3}, \quad C_3(F, \mathfrak{B}) := \frac{6 \mathfrak{B}^2-6 \mathfrak{B} F-F^2}{F^2}.
    \end{align*}
    
    \item[(b)] its Cartan tensor $\widetilde{\mathbf{T}}$ is given by
\begin{align*}
    2\widetilde{\mathbf{T}}(\overline{X},\overline{Y},\overline{Z}) &= 2C_1(F, \mathfrak{B})\,\,T(\overline{X},\overline{Y},\overline{Z}) \\
    &\quad -\frac{\mathfrak{B}\,C_2(F, \mathfrak{B})}{F^2}\,\set{\hbar(\overline{X},\overline{Z})\,\ell(\overline{Y}) + \hbar(\overline{Y},\overline{Z})\,\ell(\overline{X})} \\
    &\quad +\frac{C_2(F, \mathfrak{B})}{F}\set{\mathbf{B}(\overline{X})\,\hbar(\overline{Y},\overline{Z}) + \mathbf{B}(\overline{Y})\,\hbar(\overline{X},\overline{Z})} \\
    &\quad + \left({D^\circ_{\gamma \overline{Z}}}C_1(F, \mathfrak{B})\right)\,g(\overline{X},\overline{Y})  + \left({D^\circ_{\gamma \overline{Z}}}C_3(F, \mathfrak{B})\right)\,\mathbf{B}(\overline{X})\,\mathbf{B}(\overline{Y}) \\
    &\quad - \left({D^\circ_{\gamma \overline{Z}}}\left[ \frac{\mathfrak{B}}{F}C_2(F, \mathfrak{B}) \right]\right)\,\ell(\overline{X})\,\ell(\overline{Y}) \\
    &\quad + \left({D^\circ_{\gamma \overline{Z}}}C_2(F, \mathfrak{B})\right)\set{\mathbf{B}(\overline{X})\,\ell(\overline{Y})+\mathbf{B}(\overline{Y})\,\ell(\overline{X})}.
\end{align*}
\end{description}
\end{prop}
\begin{proof} 
\begin{description}
    \item[(a)] The fundamental Finsler metric tensor $\widetilde{g}$ of the golden Finsler metric $\widetilde{F}$ \eqref{change} can be obtained from the relation $\widetilde{g} = \widetilde{\hbar} + \widetilde{\ell} \otimes \widetilde{\ell}$.\\
    Now, from Lemma \ref{B}, and the facts $\rho \circ \gamma = 0$ and $\rho \circ \beta = \rho \circ \widetilde{\beta} = Id$, it follows that
    \begin{align}\label{ell}
    \widetilde{\ell}(\overline{X}) &= d_J \widetilde{F}(\widetilde{\beta} \overline{X}) = d_J \widetilde{F}({\beta} \overline{X}) = \frac{\partial \widetilde{F}}{\partial F}\, d_J F(\beta \overline{X}) + \frac{\partial \widetilde{F}}{\partial \mathfrak{B}} \, d_J \mathfrak{B}(\beta \overline{X}) \nonumber \\
    &= \frac{-\mathfrak{B}^2-F^2}{F^2} \,\ell(\overline{X}) + \frac{2 \mathfrak{B}-F}{F}\, \mathbf{B}(\overline{X}).
    \end{align}
    Then, applying \eqref{ell}, Lemma \ref{B}, and Lemma \ref{B1}, one can show that
    \begin{align*}
    \widetilde{\hbar}(\overline{X},\overline{Y}) &= \widetilde{F}({{\widetilde{D}}}^{\circ}_{\gamma \overline{X}} \,\widetilde{\ell})(\overline{Y}) = \widetilde{F}({{D}}^{\circ}_{\gamma \overline{X}} \,\widetilde{\ell})(\overline{Y}) \\
    &= \widetilde{F}\, {{D}}^{\circ}_{\gamma \overline{X}} \left\{\frac{-\mathfrak{B}^2-F^2}{F^2} \,\ell(\overline{Y}) + \frac{2 \mathfrak{B}-F}{F}\, \mathbf{B}(\overline{Y})\right\} \\
    &= \widetilde{F}\, \left\{\left({{D}}^{\circ}_{\gamma \overline{X}}\left(\frac{-\mathfrak{B}^2-F^2}{F^2}\right)\right)\,\ell(\overline{Y}) + \left({{D}}^{\circ}_{\gamma \overline{X}}\left(\frac{2 \mathfrak{B}-F}{F}\right)\right)\, \mathbf{B}(\overline{Y})\right\} \\
    &\quad + \widetilde{F}\,\left\{\left(\frac{-\mathfrak{B}^2-F^2}{F^2}\right)\,({{D}}^{\circ}_{\gamma \overline{X}}\ell)(\overline{Y}) + \left(\frac{2 \mathfrak{B}-F}{F}\right)\, ({{D}}^{\circ}_{\gamma \overline{X}}\mathbf{B})(\overline{Y})\right\} \\
    &= \left(\frac{\mathfrak{B}^2-\mathfrak{B} F-F^2}{F}\right)\, \Bigg{\{} \left(\frac{2\mathfrak{B}^2}{F^3}\,\ell(\overline{X}) - \frac{2 \mathfrak{B}}{F^2}\, \mathbf{B}(\overline{X})\right)\,\ell(\overline{Y}) \\
    &\quad + \left(-\frac{2 \mathfrak{B}}{F^2}\,\ell(\overline{X}) + \frac{2}{F}\, \mathbf{B}(\overline{X})\right)\, \mathbf{B}(\overline{Y}) \Bigg{\}} \\
    &\quad + \left(\frac{\mathfrak{B}^2-\mathfrak{B} F-F^2}{F}\right)\,\left(\frac{-\mathfrak{B}^2-F^2}{F^2}\right)\,F^{-1} \,\hbar(\overline{X},\overline{Y}).
    \end{align*}
    Hence, the fundamental Finsler metric tensor $\widetilde{g}$ is obtained by 
    \begin{align*}
    \widetilde{g}(\overline{X},\overline{Y}) &= \widetilde{\hbar}(\overline{X},\overline{Y}) + (\widetilde{\ell} \otimes \widetilde{\ell}) (\overline{X},\overline{Y}) \\ 
    &= \frac{\left(\mathfrak{B}^2+F^2\right) \left(-\mathfrak{B}^2+\mathfrak{B} F+F^2\right)}{F^4}\hbar(\overline{X},\overline{Y}) \\ 
    &\quad + \frac{2 \left(\mathfrak{B}^2-\mathfrak{B} F-F^2\right)}{F^2}\,\mathbf{B}(\overline{X})\mathbf{B}(\overline{Y}) \\
    &\quad + \frac{2 \mathfrak{B}^2 \left(\mathfrak{B}^2-\mathfrak{B} F-F^2\right)}{F^4}\,\ell(\overline{X})\,\ell(\overline{Y}) \\
    &\quad - \frac{2 \mathfrak{B} \left(\mathfrak{B}^2-\mathfrak{B} F-F^2\right)}{F^3}\,\left\{\mathbf{B}(\overline{X})\,\ell(\overline{Y})+\mathbf{B}(\overline{Y})\,\ell(\overline{X})\right\} \\ 
    &\quad + \left(\frac{\mathfrak{B}^2+F^2}{F^2} \,\ell(\overline{X}) - \frac{2 \mathfrak{B}-F}{F}\, \mathbf{B}(\overline{X})\right)\\
    & \qquad \times \left(\frac{\mathfrak{B}^2+F^2}{F^2} \,\ell(\overline{Y}) - \frac{2 \mathfrak{B}-F}{F}\, \mathbf{B}(\overline{Y})\right) \\ 
    &= C_1(F, \mathfrak{B})\,g(\overline{X},\overline{Y}) + C_3(F, \mathfrak{B})\,\mathbf{B}(\overline{X})\,\mathbf{B}(\overline{Y}) \\
    &\quad - \frac{\mathfrak{B}}{F} C_2(F, \mathfrak{B})\,\ell(\overline{X})\,\ell(\overline{Y}) \\
    &\quad + C_2(F, \mathfrak{B})\,\left\{\mathbf{B}(\overline{X})\,\ell(\overline{Y})+\mathbf{B}(\overline{Y})\,\ell(\overline{X})\right\}.
    \end{align*}
    
    \item[(b)] To derive the Cartan tensor $\widetilde{\mathbf{T}}$, we use the simplified expression of the fundamental metric $\widetilde{g}$ obtained above and apply the identity $2\widetilde{\mathbf{T}}(\overline{X},\overline{Y},\overline{Z}) = ({D^\circ_{\gamma \overline{Z}}}\widetilde{g})(\overline{X},\overline{Y})$. By distributing the covariant derivative over the terms, applying the Leibniz rule, and recognizing the base relations $({D^\circ_{\gamma \overline{Z}}}g)(\overline{X},\overline{Y}) = 2\mathbf{T}(\overline{X},\overline{Y},\overline{Z})$ and $(D^\circ_{\gamma \overline{Z}}\ell)(\overline{Y}) = \frac{1}{F}\hbar (\overline{Z}, \overline{Y})$, the expression collapses naturally into the stated formula.
\end{description}
\end{proof}
\begin{thm}
The fundamental golden Finsler metric tensor $\widetilde{g}$ of $\widetilde{F}$ has maximal rank if and only if
\begin{equation}
\label{Eq:g_tilde_non_degenerate}
F^2(1-2A^2)+3\mathfrak{B}^2 \neq 0.
\end{equation}
\end{thm}

\begin{proof}
Let $\widetilde{g}$, given in Proposition \ref{hh1}(a), be the Finsler metric tensor associated with the golden Finsler metric $\widetilde{F}$. Suppose that $\widetilde{g}(\overline{W}, \overline{Z}) = 0$ for all $\overline{W}\in\cp$. 

That is, utilizing the substitution constants $C_1, C_2,$ and $C_3$ from Proposition \ref{hh1}, we have:
\begin{align*}
0 &= C_1(F,\mathfrak{B})\,g(\overline{W},\overline{Z})
+ C_3(F, \mathfrak{B})\,\mathbf{B}(\overline{W})\,\mathbf{B}(\overline{Z}) -\frac{\mathfrak{B}}{F}C_2(F, \mathfrak{B})\,\ell(\overline{W})\,\ell(\overline{Z}) \\
&\quad +C_2(F, \mathfrak{B})\,\left\{\mathbf{B}(\overline{W})\,\ell(\overline{Z})+\mathbf{B}(\overline{Z})\,\ell(\overline{W})\right\}.
\end{align*}
Thus, by substituting $\overline{W} = \overline{A}$, and using the identities $g(\overline{A},\overline{Z}) = \mathbf{B}(\overline{Z})$, $\ell(\overline{A}) = \frac{\mathfrak{B}}{F}$, and $\mathbf{B}(\overline{A}) = g(\overline{A},\overline{A}) := A^2$, we get
\begin{equation}\label{12q}
\zeta_1 \,\ell(\overline{Z}) + \zeta_2\, \mathbf{B}({\overline{Z}}) = 0,
\end{equation}
where
\begin{align*}
\zeta_1 &:= C_2(F, \mathfrak{B})\left(A^2 - \frac{\mathfrak{B}^2}{F^2}\right) = \frac{\left(-4 \mathfrak{B}^3+3 \mathfrak{B}^2 F+F^3\right) \left(F^2 A^2-\mathfrak{B}^2\right)}{F^5}, \\
\zeta_2 &:= C_1(F, \mathfrak{B}) + A^2 C_3(F, \mathfrak{B}) + \frac{\mathfrak{B}}{F}C_2(F, \mathfrak{B}) \\
&= \frac{-5 \mathfrak{B}^4+4 \mathfrak{B}^3 F-F^2 A^2 \left(-6 \mathfrak{B}^2+6 \mathfrak{B} F+F^2\right)+2 \mathfrak{B} F^3+F^4}{F^4}.
\end{align*}

Similarly, substituting $\overline{W} = \overline{\eta}$, along with using $g(\overline{\eta}, \overline{Z}) = F \ell(\overline{Z})$, $\ell(\overline{\eta}) = F$, and $\mathbf{B}(\overline{\eta}) = \mathfrak{B}$, we obtain
\begin{equation}\label{13q}
\zeta_3\, \ell(\overline{Z}) + \zeta_4\, \mathbf{B}({\overline{Z}}) = 0,
\end{equation}
with
\begin{align*}
\zeta_3 &:= F\,C_1(F, \mathfrak{B}) = \frac{-\mathfrak{B}^4+\mathfrak{B}^3 F+\mathfrak{B} F^3+F^4}{F^3}, \\
\zeta_4 &:= \mathfrak{B}\,C_3(F, \mathfrak{B}) + F\,C_2(F, \mathfrak{B}) = \frac{2 \mathfrak{B}^3-3 \mathfrak{B}^2 F-\mathfrak{B} F^2+F^3}{F^2}.
\end{align*}

Hence, the system of linear equations \eqref{12q} and \eqref{13q} in terms of $\ell(\overline{Z})$ and $\mathbf{B}({\overline{Z}})$ has a non-trivial solution if and only if its determinant vanishes ($\zeta_1 \zeta_4 - \zeta_2 \zeta_3 = 0$), which yields:
$$
\frac{\left(\mathfrak{B}^2-\mathfrak{B} F-F^2\right)^3 \left(3 \mathfrak{B}^2+F^2 \left(1-2 A^2\right)\right)}{F^7} = 0.
$$
Notice that $\mathfrak{B}^2-\mathfrak{B} F-F^2 = F^2\left(\left(\frac{\mathfrak{B}}{F}\right)^2 - \frac{\mathfrak{B}}{F} - 1\right) = F \widetilde{F}$. As $\widetilde{F} \neq 0$ over $\T M$, this factor is strictly non-vanishing. Thus, we conclude that
$$
F^2(1-2A^2)+3\mathfrak{B}^2 = 0.
$$
Therefore, the system has only the trivial solution $\ell(\overline{Z}) = \mathbf{B}({\overline{Z}}) = 0$ if and only if the base Finsler structure $F$ and the $\pi$-form $\mathfrak{B}$ satisfy the condition
$$
F^2(1-2A^2)+3\mathfrak{B}^2 \neq 0.
$$
Under this condition, substituting $\ell(\overline{Z}) = \mathbf{B}({\overline{Z}}) = 0$ back into the initial equation $\widetilde{g}(\overline{W}, \overline{Z}) = 0$ leaves only $C_1(F, \mathfrak{B})g(\overline{W}, \overline{Z}) = 0$. Because $C_1 \neq 0$ and the base metric $g$ is non-degenerate, one immediately deduces that $\overline{Z} = 0$. Therefore, $\widetilde{g}$ is non-degenerate (i.e., has maximal rank) if and only if condition \eqref{Eq:g_tilde_non_degenerate} is satisfied. 
\end{proof}

\begin{rem}
From now on, we consider the golden Finsler metric $\widetilde{F}$ defined by \eqref{change} to satisfy condition \eqref{Eq:g_tilde_non_degenerate}.
\end{rem}
\subsection{Sprays and Connections of Golden Metrics}
To avoid complicated formulas and to determine the geodesic spray of the golden Finsler metric, we restrict our further investigation to a special $1$-form, namely, a concurrent $\pi$-vector field.  
A non-vanishing $\pi$-vector field $\overline{A}$ is called a concurrent $\pi$-vector field if it satisfies \cite{r94a} for all $\overline{Z} \in \cp$:
\begin{equation}\label{concurrent}
    \nabla_{\beta \overline{Z}}\,\overline{A} = - \overline{Z} = {{D}}^{\circ}_{\beta \overline{Z}}\,\overline{A}, \qquad
    \nabla_{\gamma \overline{Z}}\,\overline{A} = 0 = {{D}}^{\circ}_{\gamma \overline{Z}}\,\overline{A}.
\end{equation}
Moreover, if $\mathbf{B}$ is the $\pi$-form associated with a concurrent $\pi$-vector field $\overline{A}$ (that is, $\mathbf{B}=i_{\overline{A}}\,g$), then $\mathbf{B}$ satisfies for all $\overline{Z},\, \overline{W} \in \cp$:
$$ 
(\nabla_{\beta \overline{Z}}\mathbf{B})(\overline{W}) = -g(\overline{Z},\overline{W}) = ({{D}}^{\circ}_{\beta \overline{Z}}\mathbf{B})(\overline{W}), \qquad (\nabla_{\gamma\overline{Z}}\mathbf{B})(\overline{W}) = 0 = ({{D}}^{\circ}_{\gamma \overline{Z}}\mathbf{B})(\overline{W}).
$$

\begin{rem}
\begin{itemize}
    \item[$\bullet$] If the $\pi$-vector field $\overline{A}$ associated with the given scalar $\pi$-form $\mathbf{B}$ is concurrent over $(M,F)$, then $\widetilde{F}$ will be called a special golden Finsler metric.
    \item[$\bullet$] It is known that a concurrent $\pi$-vector field $\overline{A}$ and its associated $\pi$-form $\textbf{B}$ are independent of the directional argument $y$ \cite{r94a}.
\end{itemize}
\end{rem}

We now determine the relationship between the geodesic spray $\widetilde{G}$ of the special golden Finsler metric $\widetilde{F}$ and the geodesic spray $G$ of $F$. 

\begin{thm}\label{th.22} 
The canonical spray $\widetilde{G}$ of the special golden Finsler metric can be expressed as
\begin{equation}\label{Eq:canonical_spray}
\widetilde{G} = G - \mathcal{E}_1 \,\mathcal{C} - \mathcal{E}_2 \,\gamma \overline{A},
\end{equation}
where $\mathcal{E}_1$ and $\mathcal{E}_2$ are smooth functions on $TM$ defined by 
\begin{equation*}
\mathcal{E}_1 := \frac{F^2 \left(F^3 + 3 F \mathfrak{B}^2 - 4 \mathfrak{B}^3\right)}{\left(\mathfrak{B}^2-\mathfrak{B} F-F^2\right)\mathcal{Z}}, \qquad \mathcal{E}_2 := \frac{2 F^4}{\mathcal{Z}},
\end{equation*}
with $\mathcal{Z} := F^2 \left(1-2 A^2\right) + 3 \mathfrak{B}^2$.
\end{thm}

\begin{proof} 
Let $(M,F)$ be a Finsler manifold. Let $\widetilde{F}$ be the special golden Finsler metric \eqref{change} with a concurrent $\pi$-vector field $\overline{A}$ and its concurrent scalar $\pi$-form $\mathbf{B}$ over $(M,F)$. Thus, the exterior $\pi$-form of $\widetilde{F}$ is defined by $\widetilde{\Omega} = \frac{1}{2}dd_{{J}}\,\widetilde{F}^2$. As the difference between the two geodesic sprays of $F$ and $\widetilde{F}$ is a vertical vector field (i.e., $\widetilde{G} = G+\gamma \overline{V}$, for some $\pi$-vector field $\overline{V}$), we get
\begin{equation}\label{ch5eq}
-\frac{1}{2}\,d \widetilde{F}^2(X) = i_{\widetilde{G}}\,\widetilde{\Omega}(X) = \frac{1}{2}i_{G}\,dd_{J}\widetilde{F}^2(X) + \frac{1}{2}i_{\gamma \overline{V}}\,dd_{J}\widetilde{F}^2(X).
\end{equation}
Using Lemmas \ref{B1} and \ref{B} together with 
$$
\beta \overline{\eta} = G, \quad X = hX+vX = \beta\rho X+\gamma K X, \quad d \mathfrak{B}(G) = -F^2,
$$ 
we obtain
\begin{align*}
\frac{1}{2} d\widetilde{F}^2(X) &= \widetilde{F}\,d\widetilde{F}(X) \\
&= \left(\frac{\mathfrak{B}^2-\mathfrak{B} F-F^2}{F}\right) \set{\left(\frac{-\mathfrak{B}^2-F^2}{F^2}\right)\,dF(X)+\left(\frac{2 \mathfrak{B}-F}{F}\right)\, d\mathfrak{B}(X)} \\
&= \frac{-\mathfrak{B}^4+\mathfrak{B}^3 F+\mathfrak{B} F^3+F^4}{F^3}\,dF(X) + \frac{2 \mathfrak{B}^3-3 \mathfrak{B}^2 F-\mathfrak{B} F^2+F^3}{F^2}\, d\mathfrak{B}(X).
\end{align*}
Also, we have
\begin{align*}
\frac{1}{2}\,i_{G}\,dd_{J}\widetilde{F}^{2}(X) &= \frac{1}{2}\{dd_{J}\widetilde{F}^{2}(\beta \overline{\eta} ,X)\} \\
&= \frac{1}{2} \set{G \cdot d_J\widetilde{F}^{2}(X) - X \cdot d_J\widetilde{F}^{2}(G)-d_J\widetilde{F}^{2}[G,X]} \\
&= \frac{1}{2} \set{G \cdot (2\widetilde{F} \widetilde{\ell}(\rho X)) - X \cdot(2 \widetilde{F} \widetilde{\ell}(\overline{\eta})) - 2\widetilde{F}\widetilde{\ell}(\rho[G,X])} \\
&= (G \cdot \widetilde{F})\,\widetilde{\ell}(\rho X) + \widetilde{F} (G \cdot \widetilde{\ell}(\rho X)) - (X \cdot \widetilde{F}^2) - \widetilde{F}\,\widetilde{\ell}(\rho[G,X]) \\
&= d\widetilde{F}(G)\,\widetilde{\ell}(\rho X) + \widetilde{F} (G \cdot \widetilde{\ell}(\rho X)) - 2\widetilde{F}(X \cdot \widetilde{F}) - \widetilde{F}\,\widetilde{\ell}(\rho[G,X]) \\
&= (F^2-2 \mathfrak{B}F)\,\widetilde{\ell}(\rho X) \\
&\quad + \widetilde{F}\left( G \cdot \widetilde{\ell}(\rho X) + 2 \left(\frac{\mathfrak{B}^2+F^2}{F^2}\right)\,dF(X) \right). \\
& \quad +\widetilde{F} \left(\frac{2 \mathfrak{B}-F}{F}\right)\, d\mathfrak{B}(X) +\widetilde{F}\, \widetilde{\ell}(\rho[G,X]).
\end{align*}
Using Lemma \ref{B} together with the following relations:
\begin{align*}
\rho[G,X] &= \rho[G,hX+vX] = D^{\circ}_G \rho X-K X, \\
(D^{\circ}_{G} \,\mathbf{B})(\overline{X}) &= -g(\overline{X},\overline{\eta}) = -F\,\ell(\overline{X}), \quad (D^{\circ}_{G}\, \ell)(\overline{X}) = (\nabla_{G}\, \ell)(\overline{X}) = 0, \\
d\mathfrak{B}(X) &= \mathbf{B}(K X) - F\ell({\rho X}), \quad dF(X) = dF(\gamma K X) = \ell(K X),
\end{align*}
we get {\small{
\begin{align*}
\frac{1}{2}\,i_{G}\,dd_{J}\widetilde{F}^{2}(X) &= (F^2-2 \mathfrak{B}F)\,\left(\left(\frac{-\mathfrak{B}^2-F^2}{F^2}\right)\,\ell(\rho {X})+\left(\frac{2 \mathfrak{B}-F}{F}\right)\, \mathbf{B}(\rho {X})\right) \\
&\quad + \left(\frac{\mathfrak{B}^2-\mathfrak{B} F-F^2}{F}\right)\, G \cdot\left(\left(\frac{-\mathfrak{B}^2-F^2}{F^2}\right)\,\ell(\rho {X})+\left(\frac{2 \mathfrak{B}-F}{F}\right)\, \mathbf{B}(\rho {X})\right) \\
&\quad - 2\left(\frac{\mathfrak{B}^2-\mathfrak{B} F-F^2}{F}\right)\,\left(\left(\frac{-\mathfrak{B}^2-F^2}{F^2}\right)\,dF(X)+\left(\frac{2 \mathfrak{B}-F}{F}\right)\, d\mathfrak{B}(X)\right) \\
&\quad - \left(\frac{\mathfrak{B}^2-\mathfrak{B} F-F^2}{F}\right)\,\left(\left(\frac{-\mathfrak{B}^2-F^2}{F^2}\right)\,\ell(\rho[G,X])+\left(\frac{2 \mathfrak{B}-F}{F}\right)\, \mathbf{B}(\rho[G,X])\right) \\
&= \left(\frac{4 \mathfrak{B}^3-3 \mathfrak{B}^2 F-F^3}{F}\right) \ell(\rho {X}) - \left(6 \mathfrak{B}^2-6 \mathfrak{B} F-F^2\right)\, \mathbf{B}(\rho {X}) \\
&\quad + \left(\frac{\mathfrak{B}^4-\mathfrak{B}^3 F-\mathfrak{B} F^3-F^4}{F^3}\right) dF(X) - \left(\frac{2 \mathfrak{B}^3-3 \mathfrak{B}^2 F-\mathfrak{B} F^2+F^3}{F^2}\right) d\mathfrak{B}(X).
\end{align*}}}
On the other hand,  we have {\small{
\begin{align*}
\frac{1}{2}\,i_{\gamma \overline{V}}\,dd_{J}\widetilde{F}^{2}(X) &= \widetilde{g}(\overline{V},\rho X) \\
&= \frac{\left(\mathfrak{B}^2+F^2\right) \left(-\mathfrak{B}^2+\mathfrak{B} F+F^2\right)}{F^4}\,g(\overline{V},\rho {X}) + \frac{6 \mathfrak{B}^2-6 \mathfrak{B} F-F^2}{F^2}\,\mathbf{B}(\overline{V})\,\mathbf{B}(\rho {X}) \\
&\quad + \frac{4 \mathfrak{B}^4-3 \mathfrak{B}^3 F-\mathfrak{B} F^3}{F^4}\,\ell(\overline{V})\,\ell(\rho {X}) \\
&\quad + \frac{F^3-4 \mathfrak{B}^3+3 \mathfrak{B}^2 F}{F^3}\,\set{\mathbf{B}(\overline{V})\,\ell(\rho {X})+\mathbf{B}(\rho {X})\,\ell(\overline{V})}.
\end{align*}}}
Substituting the last two relations into Equation \eqref{ch5eq} yields
\begin{align*}
&\left(\frac{\mathfrak{B}^4-\mathfrak{B}^3 F-\mathfrak{B} F^3-F^4}{F^3}\right) dF(X) - \left(\frac{2 \mathfrak{B}^3-3 \mathfrak{B}^2 F-\mathfrak{B} F^2+F^3}{F^2}\right) d\mathfrak{B}(X) \\
&= \left(\frac{4 \mathfrak{B}^3-3 \mathfrak{B}^2 F-F^3}{F}\right) \ell(\rho {X}) - \left(6 \mathfrak{B}^2-6 \mathfrak{B} F-F^2\right)\, \mathbf{B}(\rho {X}) \\
&\quad + \left(\frac{\mathfrak{B}^4-\mathfrak{B}^3 F-\mathfrak{B} F^3-F^4}{F^3}\right) dF(X) - \left(\frac{2 \mathfrak{B}^3-3 \mathfrak{B}^2 F-\mathfrak{B} F^2+F^3}{F^2}\right) d\mathfrak{B}(X) \\
&\quad + \frac{\left(\mathfrak{B}^2+F^2\right) \left(-\mathfrak{B}^2+\mathfrak{B} F+F^2\right)}{F^4}\,g(\overline{V},\rho {X}) + \frac{6 \mathfrak{B}^2-6 \mathfrak{B} F-F^2}{F^2}\,\mathbf{B}(\overline{V})\,\mathbf{B}(\rho {X}) \\
&\quad + \frac{4 \mathfrak{B}^4-3 \mathfrak{B}^3 F-\mathfrak{B} F^3}{F^4}\,\ell(\overline{V})\,\ell(\rho {X}) \\
&\quad + \frac{F^3-4 \mathfrak{B}^3+3 \mathfrak{B}^2 F}{F^3}\,\set{\mathbf{B}(\overline{V})\,\ell(\rho {X})+\mathbf{B}(\rho {X})\,\ell(\overline{V})}.
\end{align*}
Since the Finsler metric tensor $g$ is non-degenerate, the above relation reduces to
\begin{align}\label{ch52.eq.5}
&\frac{\left(\mathfrak{B}^2+F^2\right) \left(-\mathfrak{B}^2+\mathfrak{B} F+F^2\right)}{F^4}\,\overline{V} \nonumber \\
&= \Bigg{\{}\frac{-4 \mathfrak{B}^3+3 \mathfrak{B}^2 F+F^3}{F^2} + \frac{4 \mathfrak{B}^3-3 \mathfrak{B}^2 F-F^3}{F^4} \,\mathbf{B}({\overline{V}}) - \frac{4 \mathfrak{B}^4-3 \mathfrak{B}^3 F-\mathfrak{B} F^3}{F^5}\,\ell(\overline{V})\Bigg{\}}\overline{\eta} \nonumber \\
&\quad + \Bigg{\{}6 \mathfrak{B}^2-6 \mathfrak{B} F-F^2 - \frac{6 \mathfrak{B}^2-6 \mathfrak{B} F-F^2}{F^2}\,\mathbf{B}(\overline{V}) + \frac{4 \mathfrak{B}^3-3 \mathfrak{B}^2 F-F^3}{F^3}\,\ell(\overline{V})\Bigg{\}}\overline{A},
\end{align}
where $\ell(\overline{V})$ and $\mathbf{B}(\overline{V})$ are geometric quantities given by the following system:
\begin{equation}\label{mm}
\begin{aligned}
    a_1 \,\ell(\overline{V})+b_1 \,\mathbf{B}(\overline{V}) &= c_1, \\
    a_2 \,\ell(\overline{V})+b_2 \,\mathbf{B}(\overline{V}) &= c_2,
\end{aligned}
\end{equation}
with coefficients determined by
\begin{align*}
a_1 &:= \frac{-\mathfrak{B}^4+\mathfrak{B}^3 F+\mathfrak{B} F^3+F^4}{F^4}, \quad b_1 := \frac{2 \mathfrak{B}^3-3 \mathfrak{B}^2 F-\mathfrak{B} F^2+F^3}{F^3}, \quad c_1 := F^2 \, b_1, \\
a_2 &:= \frac{\left(-4 \mathfrak{B}^3+3 \mathfrak{B}^2 F+F^3\right) \left(F^2 A^2-\mathfrak{B}^2\right)}{F^5}, \\
b_2 &:= \frac{-5 \mathfrak{B}^4+4 \mathfrak{B}^3 F+6 \mathfrak{B}^2 F^2 A^2+2 \mathfrak{B} F^3 \left(1-3 A^2\right)-F^4 \left(A^2-1\right)}{F^4}, \\
c_2 &:= \frac{-4 \mathfrak{B}^4+3 \mathfrak{B}^3 F+6 \mathfrak{B}^2 F^2 A^2-6 \mathfrak{B} F^3 A^2+\mathfrak{B} F^3-F^4 A^2}{F^2}.
\end{align*}
Using condition \eqref{Eq:g_tilde_non_degenerate}, the system \eqref{mm} yields the unique solution:
\begin{align*}
\ell(\overline{V}) &= \frac{F^3 (2 \mathfrak{B}-F) \left(\mathfrak{B}^2+F^2\right)}{\left(\mathfrak{B}^2-\mathfrak{B} F-F^2\right) \left(3 \mathfrak{B}^2+F^2 \left(1-2 A^2\right)\right)}, \\
\mathbf{B}(\overline{V}) &= \frac{F^2 \left(4 \mathfrak{B}^4-3 \mathfrak{B}^3 F-2 \mathfrak{B}^2 F^2 A^2+\mathfrak{B} F^3 \left(2 A^2-1\right)+2 F^4 A^2\right)}{\left(\mathfrak{B}^2-\mathfrak{B} F-F^2\right) \left(3 \mathfrak{B}^2+F^2 \left(1-2 A^2\right)\right)}.
\end{align*}

Consequently, the canonical sprays $G$ and $\widetilde{G}$ are related by:
$$
\widetilde{G} = G-\frac{F^2 \left(-4 \mathfrak{B}^3+3 \mathfrak{B}^2 F+F^3\right)}{\left(\mathfrak{B}^2-\mathfrak{B} F-F^2\right) \left(3 \mathfrak{B}^2+F^2 \left(1-2 A^2\right)\right)} \,{\mathcal{C}} -\frac{2 F^4}{3 \mathfrak{B}^2+F^2 \left(1-2 A^2\right)}\,\gamma \overline{A}.
$$
This completes the proof.
\end{proof}

\begin{thm}\label{projectively related sprays}
Let $(M,F)$ be a Finsler manifold endowed with a concurrent $\pi$-vector field $\overline{A}$. Under the special golden Finsler metric change \eqref{change}, the Finsler metrics $F$ and $\widetilde{F}$ do not share the same set of geodesics. (That is, the unparameterized geodesics of the base metric $F$ and the generalized golden metric $\widetilde{F}$ are strictly distinct.)
\end{thm}

\begin{proof} 
We prove this by showing that the geodesic sprays $G$ and $\widetilde{G}$ of $F$ and $\widetilde{F}$ can never be projectively related.

From Theorem \ref{th.22}, the canonical sprays $\widetilde{G}$ and $G$ are related by:
$$
\widetilde{G} = G - \mathcal{E}_1 \mathcal{C} - \mathcal{E}_2 \gamma \overline{A}
$$

By definition, if $\widetilde{G}$ and $G$ were projectively related, their difference would be proportional solely to the Liouville vector field $\mathcal{C}$. Since the vector field $\overline{A}$ is non-vanishing, this projective equivalence strictly requires the coefficient of $\overline{A}$ to vanish. That is:
$$
\mathcal{E}_2 = \frac{2 F^4}{3 \mathfrak{B}^2+F^2 \left(1-2 A^2\right)} = 0
$$

However, this condition implies $F=0$, which is impossible since $F$ is a strictly positive Finsler metric. Therefore, the sprays are never projectively related, concluding the proof.
\end{proof}

\begin{thm}\label{connection} 
Let $(M,F)$ be a Finsler manifold endowed with a concurrent $\pi$-vector field $\overline{A}$. Under the special golden Finsler metric change {\em \eqref{change}}, we have: 
\begin{description}
\item[(1)] The Barthel connections $\widetilde{\Gamma}$ and $\Gamma$ of $F$ and $\widetilde{F}$ are related by
\begin{equation}\label{nonlinear connection}
    \widetilde{\Gamma} = \Gamma + \mathbb{E}, \qquad \mathbb{E} := - \mathcal{E}_1\,J-d_J \mathcal{E}_1 \otimes \gamma \overline{\eta} - d_J \mathcal{E}_2 \otimes \gamma \overline{A}.
\end{equation}

\item[(2)] The horizontal map $\widetilde{\beta}$ associated with $\widetilde{F}$ has the form
$$
\widetilde{\beta}\,\overline{X} = {\beta}\overline{X} - {\frac 1 2}\left\{\mathcal{E}_1\,\gamma \overline{X}+d_J \mathcal{E}_1(\beta\overline{X}) \, \gamma \overline{\eta} - d_J \mathcal{E}_2 (\beta\overline{X})\, \gamma \overline{A}\right\}.
$$

\item[(3)] The horizontal projections $\widetilde{h}, h$ and vertical projections $\widetilde{v}, v$ of $F$ and $\widetilde{F}$ are related, respectively, by
$$
\widetilde{h} = h+\frac{1}{2} \mathbb{E}, \quad \widetilde{v} = v-\frac{1}{2}\mathbb{E}.
$$
 
\item[(4)] The Barthel curvature tensors $\widetilde{\Re}$ and $\Re$ of $F$ and $\widetilde{F}$ are determined by
\[ 
\widetilde{\Re} = \Re -[h,\mathbb{E}] -\frac{1}{2} [\mathbb{E},\mathbb{E}].
\]

\item[(5)] The horizontal counterpart of the Berwald connection associated with $\widetilde{F}$ has the form
\begin{align*}
{{\widetilde{D}}}^{\circ}_{ \widetilde{\beta}\, \overline{X}} \, {\overline{Y}} &= {D^\circ}_{{\beta} \overline{X}} \overline{Y} - \frac{1}{2}\left\{ \mathcal{E}_1\,D^\circ_{\gamma \overline{X}}\,\overline{Y}+d_J \mathcal{E}_1 ({\beta} \overline{X})\, D^\circ_{\gamma \overline{\eta}}\, \overline{Y}\right. \\
&\quad \left. - d_J \mathcal{E}_1 ({\beta} \overline{X})\, \overline{Y}- d_J \mathcal{E}_1 (\beta \overline{Y})\,\overline{X} - d_J \mathcal{E}_2 ({\beta} \overline{X})\,D^\circ_{\gamma \overline{A}}\, \overline{Y} \right\} \\
&\quad + \frac{1}{2}\left\{dd_J \mathcal{E}_1 (\gamma \overline{Y},{\beta} \overline{X})\, \overline{\eta} - dd_J \mathcal{E}_2 (\gamma \overline{Y}, {\beta} \overline{X}) \, \overline{A}\right\}.
\end{align*}

\item[(6)] In general, the $\pi$-vector field $\overline{A}$ does not preserve its concurrency with respect to the generalized Finsler metric $\widetilde{F}$. However, a necessary and sufficient condition for $\overline{A}$ to be concurrent with respect to $\widetilde{F}$ is that:
\begin{equation}\label{preserved concurrent}
dd_J \mathcal{E}_1 (\gamma \overline{A},{\beta} \overline{X})\, \overline{\eta} - d_J \mathcal{E}_1 (\beta \overline{A}) \overline{X} + d_J \mathcal{E}_1 (\beta \overline{X})\overline{A} - dd_J \mathcal{E}_2 (\gamma \overline{A}, {\beta} \overline{X}) \overline{A} = 0.
\end{equation}
\end{description}
\end{thm}
 
\begin{proof}
Under the special golden Finsler transformation $\widetilde{F}$ of a Finsler metric $F$ \eqref{change}: 
\begin{description}
\item[(1)] It follows from 
\begin{align*}
\widetilde{\Gamma} &:= [J,\widetilde{G}] = [J, G - \mathcal{E}_1 \,{\mathcal{C}} - \mathcal{E}_2 \,\gamma \overline{A}] \\
&= [J,G] - \mathcal{E}_1[J,\mathcal{C}] - d\mathcal{E}_1\wedge i_{\mathcal{C}}\,J + d_{J}\mathcal{E}_1 \otimes \mathcal{C} - \mathcal{E}_2[J,\gamma \overline{A}] \\
& \qquad - d\mathcal{E}_2\wedge i_{\gamma\overline{A}}\,J + d_{J}\mathcal{E}_2 \otimes \gamma \overline{A} \\
&= \Gamma - \mathcal{E}_1\,J - d_J \mathcal{E}_1 \otimes \gamma \overline{\eta} - d_J \mathcal{E}_2 \otimes \gamma \overline{A}.   
\end{align*}

\item[(2)] This is a direct consequence of $\Gamma = 2\beta \circ \rho - I$.

\item[(3)] This follows from \eqref{nonlinear connection} together with the definitions of the vertical and horizontal projectors induced from the non-linear connection.

\item[(4)] This follows from the definition of the curvature of the Barthel connection $\widetilde{\Re} := -\frac{1}{2}[\widetilde{h},\widetilde{h}]$ together with the definition of the Fr\"{o}licher-Nijenhuis bracket and its properties (see, e.g., \cite{r20}). That is, 
\begin{align*}
\widetilde{\Re} &= -\frac{1}{2} [h+\mathbb{E},h+\mathbb{E}] \\
&= -\frac{1}{2} \left([h,h] + [h,\mathbb{E}] + [\mathbb{E},h] + [\mathbb{E},\mathbb{E}]\right) \\
&= \Re - [h,\mathbb{E}] - \frac{1}{2} [\mathbb{E},\mathbb{E}]. 
\end{align*}

\item[(5)] This results from Lemma \ref{B}(i) along with the facts that 
$$
[JX,JY] = J[X,JY] + J[JX,Y], \quad vJ = J, \quad Jv = 0.
$$ 
For further details, we refer to \cite{Soleiman-Taha_Mat}.  

\item[(6)] Because the $\pi$-vector field $\overline{A}$ is concurrent with respect to the Finsler metric $F$, we have
$$
{{D}}^{\circ}_{\beta \overline{X}}\,\overline{A} = - \overline{X}, \quad \text{and} \quad {{D}}^{\circ}_{\gamma \overline{X}}\,\overline{A} = 0.
$$
Now, Lemma \ref{B}(i) gives ${{\widetilde{D}}}^{\circ}_{\gamma \overline{X}}\, \overline{Y} = {{D}}^{\circ}_{\gamma \overline{X}}\, \overline{Y}.$ Thus, ${{\widetilde{D}}}^{\circ}_{\gamma \overline{X}}\, \overline{A} = 0,$ but from the above item, 
$$
{{\widetilde{D}}}^{\circ}_{\widetilde{\beta}\, \overline{X}}\,\overline{A} \neq {{D}}^{\circ}_{\beta \overline{X}}\,\overline{A}.
$$ 
If condition \eqref{preserved concurrent} is satisfied, then ${{\widetilde{D}}}^{\circ}_{\widetilde{\beta}\, \overline{X}}\,\overline{A} = - \overline{X}$ (which means $\overline{A}$ is concurrent with respect to $\widetilde{F}$).
   
Conversely, assume that $\overline{A}$ is a concurrent $\pi$-vector field with respect to $\widetilde{F}$ (i.e., $\widetilde{D^\circ}_{\gamma\overline{X}} \, \overline{A} = 0$ and ${{\widetilde{D}}}^{\circ}_{\widetilde{\beta}\, \overline{X}}\,\overline{A} = -\overline{X}$). Hence, by   
\begin{align*}
\widetilde{D^\circ}_{\widetilde{\beta}\, \overline{X}} \, {\overline{A}} &= {D^\circ}_{\beta \overline{X}} \, {\overline{A}} + \frac{1}{2}\left\{ \mathcal{E}_1\,D^\circ_{\gamma \overline{X}}\,\overline{A}+d_J \mathcal{E}_1 ({\beta} \overline{X})\, D^\circ_{\gamma \overline{\eta}}\, \overline{A} - d_J \mathcal{E}_2 ({\beta} \overline{X})\,D^\circ_{\gamma \overline{A}}\, \overline{A}\right\},
\end{align*} 
we obtain 
\[
\mathcal{E}_1\,D^\circ_{\gamma \overline{X}}\,\overline{A}+d_J \mathcal{E}_1 ({\beta} \overline{X})\, D^\circ_{\gamma \overline{\eta}}\, \overline{A} - d_J \mathcal{E}_2 ({\beta} \overline{X})\,D^\circ_{\gamma \overline{A}}\, \overline{A} = 0. \qedhere
\]
\end{description}
\end{proof}
\section{Conclusion}
We have established a new class of Finsler metrics, that is,  golden $(\alpha, \beta)$-metrics  and  provided an  investigation that covers both the local coordinate based foundations and the global coordinate-free perspective of such metrics.  In the local study, we derived the components of the golden Finsler metric tensor $g_{ij}$, its inverse $g^{ij}$, and the geodesic spray coefficients $G^i$, which provide the precise equations of motion for paths within this space.  Also,  we have proved golden $(\alpha, \beta)$-metrics   are not almost rational Finsler metrics. Only, Randers metric \cite{TT} and golden $(\alpha, \beta)$-metrics are the known examples of  non almost rational Finsler metrics. Furthermore, we have explored the projective properties of the space, successfully establishing the exact necessary and sufficient conditions under which the golden $(\alpha, \beta)$-metric is projectively flat.
\par In the intrinsic investigation of the golden Finsler transformation. We have defined the golden metric $\widetilde{F}$ via a scalar $\pi$-form. Then, we have derived its fundamental metric tensors and precise non-degeneracy conditions. Restricting our study to the special golden Finsler metric, characterized by a concurrent $\pi$-vector field $\overline{A}$, we have established the relationship between the geodesic sprays, proving that the base metric $F$ and its golden counterpart $\widetilde{F}$ are strictly never projectively related. Furthermore, we explicitly determined how fundamental non-linear structures, including the Barthel and Berwald connections, transform under this golden change. Thus, the derivation of these global geometric objects not only enriches the theoretical landscape of Finsler geometry but also supplies the rigorous mathematical machinery required for physically modeling anisotropic spacetimes and dynamical phenomena.
\par
While this is a pure mathematics paper, the geometric frameworks we have investigated are highly relevant to theoretical physics, specially modified gravity theories, such as giving models of anisotropic spacetimes. Our concurrent vector field could be mathematically mirror of cosmic expansion, while the derived non-degeneracy condition establishes the exact boundaries where models of Lorentz-violation would theoretically break down. Therefore, our results provide a rigorous structural foundation for future applications of golden Finsler manifolds in both differential geometry and theoretical physics.
\bigskip

\noindent\textbf{Declarations}
\begin{itemize}
\item {\textbf{Generative AI and AI-assisted technologies in the manuscript preparation process}}: During the preparation of this work the authors used Google Gemini in order to refine the language and improve the structural flow of the text. After using this tool, the authors reviewed and edited the content as needed and take full responsibility for the content of the published article. 
\item \textbf{Competing interests}: The authors declare no conflict of interest.
  \item \textbf{Availability of data and material}: Not applicable.
  \item \textbf{Funding}: Not applicable.
  \item \textbf{Contributions of authors}: The authors have made substantive contributions
to the article and assume full responsibility for its content. The authors read and approved the final manuscript.
\end{itemize}

\providecommand{\bysame}{\leavevmode\hbox
to3em{\hrulefill}\thinspace}
\providecommand{\MR}{\relax\ifhmode\unskip\space\fi MR }
\providecommand{\MRhref}[2]{%
  \href{http://www.ams.org/mathscinet-getitem?mr=#1}{#2}
} \providecommand{\href}[2]{#2}


\begin{thebibliography}{21}

\bibitem{Erasmo} E. Caponio, M. A. Javaloyes and M.  Sánchez, 
\emph{Wind Finslerian structures: from Zermelo’s navigation to the causality of spacetimes},  Mem. Amer. Math. Soc. 300 (1501), 134 (2024).

\bibitem{Bao}
D. Bao, \emph{On  two curvature-driven problems in Riemann-Finsler geometry}, Advanced Studies in Pure Mathematics 48, 2007
Finsler Geometry, Sapporo 2005 - In Memory of Makoto Matsumoto,  19--71.


\bibitem{[13]}
S.-S. Chern and Z. Shen, \emph{Riemann-Finsler geometry}, Nankai Tracts in Mathematics, 6, World Scientific, Hackensack, 2005. 

\bibitem{Projectively Flat Special} B.C.  Chethana and S.K. Narasimhamurthy, \emph{Locally Projectively Flat Special $(\alpha ,\beta)$-metric}, Palest.  J.  Math.  10\textbf{ (I)} (2021) , 69–74.

\bibitem{gold ratio1}
M. Crasmareanu and C.-E. Hre\c{t}canu, 
\emph{Golden differential geometry}, Chaos, Solitons \& Fractals, \textbf{38}(5), (2008), 1229-1238.

\bibitem{r20}
A.~Fr\"{o}licher and A.~Nijenhuis, \emph{Theory of vector-valued
differential forms}, {I}, Ann. Proc. Kon. Ned. Akad., A,
\textbf{{59}} (1956), 338--359.


 \bibitem{gold ratio2}
A. Gezer, N. Cengiz and A. Salimov, \emph{On integrability of Golden Riemannian structures}, Turkish J.  Math. \textbf{37} (4), (2013) 693-703.

\bibitem{r21}
J.~Grifone, \emph{Structure pr\'esque-tangente et connexions,
\textsc{I}}, Ann.
  Inst. Fourier, Grenoble,\textbf{ 22}, 1 (1972), 287-334.

\bibitem{r22}
J.~Grifone, \emph{Structure presque-tangente et connexions,
\textsc{II}}, Ann. Inst. Fourier, Grenoble, \textbf{22}, 3 (1972), 291-338.

\bibitem{gold ratio3}
C. Hretcanu  and M. Cra\c{s}mareanu, \emph{Applications of the Golden Ratio on Riemannian Manifolds}, Turkish Journal of Mathematics \textbf{33} (2), (2009)179-191. 

\bibitem{Finsler definitions} M. Javaloyes and M. Sánchez, \textit{On the definition and examples of Finsler metrics},   Ann. Sc. Norm.
Super. Pisa Cl. Sci. (5) Vol. XIII (2014), 813-858.  

\bibitem{r27}
J.~Klein and A.~Voutier, \emph{Formes ext\'{e}rieures
g\'{e}n\'{e}ratrices de
  sprays}, Ann. Inst. Fourier, Grenoble, \textbf{18}, 1 (1968), 241-260.
  
 \bibitem{Lovas} R.L. Lovas, \emph{A note on Finsler-Minkowski norms}, Houston J.  Math.  \textbf{33} (2007), 701-707.
 
 \bibitem{MM}
M. Matsumoto, \emph{On C-reducible Finsler spaces}, Tensor, N. S. \textbf{24} (1972), 29-37.

\bibitem{[7]}
M. Matsumoto, \emph{Theory of Finsler spaces with $(\alpha,\beta)$-metrics}, Rep. Math. Phys., \textbf{31} (1991), 43-47.

\bibitem{Anastaise} R. Miron and M. Anastaise,  \emph{Vector bundles and Lagrange spaces with applications to Relativity}, 1997.

\bibitem{r93}
R.~Miron and M.~Anastasiei, \emph{The geometry of Lagrange spaces:
Theory and applications}, Kluwer Acad. Publ.,  \textbf{59}, 1994.

\bibitem{shenGCY}
Z. Shen and G. Civi Yildirim, \emph{On a class of projectively flat metrics with constant flag curvature}, Canad. J. Math. \textbf{60} (2), (2008) 443–456.



\bibitem{amr3} A. Soleiman, \emph{Recurrent Finsler manifolds under projective change}, Int. J. Geom. Meth. Mod. Phys.,  \textbf{13} (2016).



\bibitem{Szilasi-book} J. Szilasi, R.L. Lovas and D.Cs. Kert\'esz,
  \emph{Connections, sprays and Finsler structures}, World Scientific, 2014.

\bibitem{Tachibana}
S. Tachibana, \emph{On Finsler spaces which admit concurrent vector field},  Tensor, N. S.,  \textbf{1} (1950), 1--5.

\bibitem{Soleiman-Taha_Mat}
A.~Soleiman and E. H. Taha, \emph{On a generalized Matsumoto metric with special $\pi$-form}, Hacet. J. Math. Stat. (Advanced Online Publication 2026). doi:10.15672/hujms.1613505

\bibitem{TT} E. H. Taha and B. Tiwari, \emph{On almost rational Finsler metrics},  Bull.  Iran.  Math.  Soc. \textbf{49} (8) (2023).  

\bibitem{Tiwari-Kumar-Tayebi} B. Tiwari, M. Kumar and A. Tayebi,  \emph{On Generalized Kropina change of generalized m-th root Finsler metrics.}, Proc. Natl. Acad. Sci., India, Sect. A Phys. Sci.  \textbf{91},  443–450 (2021).

\bibitem{r92}
N.~L. Youssef, S.~H. Abed and A.~Soleiman, \emph{Cartan and
  Berwald connections in the pullback formalism}, Algebras,
  Groups and Geometries, \textbf{25} (2008), 363-384. 


\bibitem{r94a}
N.~L. Youssef, S.~H. Abed and A.~Soleiman, \emph{Concurrent $\pi$-vector fields and energy $\beta$-change}, Int.  J.  Geom.  Meth.  Mod.  Phys.  \textbf{60} (2011), 1003–1031.

\bibitem{r86}
N.~L. Youssef, S.~H. Abed and A.~Soleiman, \emph{A global approach to the theory of special Finsler manifolds}, J.  Math.  Kyoto Univ. \textbf{48} (2008), 857-893.

\bibitem{r94}
N.~L. Youssef, S.~H. Abed and A.~Soleiman, \emph{A global approach to  the theory  of connections in Finsler geometry}, Tensor, N.  S. \textbf{71} (2009),
187-208. 

\bibitem{r96}
N.~L. Youssef, S.~H. Abed and A.~Soleiman, \emph{Geometric objects associated with the fundamental connections in Finsler geometry},   J.  Egypt.  Math.  Soc. \textbf{18} (2010), 67-90.  



 

\end{thebibliography}
\end{document}